\documentclass[twoside]{aiml20}

\usepackage{aiml20macro}

\usepackage{graphicx}
\usepackage{amsmath}
\usepackage{amssymb,latexsym,color,dsfont,enumitem}
\usepackage{geometry}
\usepackage{tikz}
\usepackage{pgf}

\def\BF{{\bf F}}

\def\BT{{\bf T}}
\def\BZ{{\bf Z}}
\def\bsq{\blacksquare}

\def\Di{\Diamond}
\def\fo{\forall}

\def\ri{\rightarrow}
\def\BK{{\bf K}}
\def\BQ{{\bf Q}}
\def\BC{{\bf C}}
\def\sqr{\square}

\def\bLa{{\bf \Lambda}}
\def\we{\wedge}

\def\sem0{\,_{\_\!\_\!\_}\!|\,}
\def\sem{\,_{\_\!\_\!\_}\!|\,\bS5}
\def\rtt{\rightthreetimes}

\def\bSL4{{\bf SL4}}

\def\<{\langle}
\def\>{\rangle}

\def\a1an{{a_1,\dots,a_n}}

\def\b1{\Box^{-1}}
\def\B1Bk{{B_1,\dots,B_k}}

\def\bbK4{{\bf K4}}
\def\bbD4{{\bf D4}}
\def\bbS4{{\bf S4}}

\def\bC{{\bf C}}

\def\bD{{\bf D}}
\def\bD4{{\bf D4}}

\def\bGL{{\bf GL}}

\def\bLa{{\bf \Lambda}}

\def\bQS4{{\bf QS4}}

\def\bS5{{\bf S5}}
\def\BS5{{\bf S5}}

\def\bvee0{\bigvee}
\def\bvee{\boldsymbol{\vee}}

\def\C12{{\cal C}_1\t{\cal C}_2}

\def\d1{\Diamond^{-1}}
\def\da1aian{{a_1,\dots,a_i,\dots,a_n}}
\def\da1ban{{a_1,\dots,b,\dots,a_n}}
\def\da1n{{a_1,\dots,a_n}}

\def\Di{\Diamond}

\def\dS5{{\text{{\bf {\em S5}}}}}

\def\dx1n{\tx{\hbox{$x_1,\dots,x_n$}}}
\def\dx1yxn{{x_1,\dots,y,\dots,x_n}}

\def\fo{\forall}

\def\fr2{{\text{\scriptsize{\frame{2}}}}}

\def\mo{\vDash}

\def\p1pk{{p_1,\dots,p_k}}

\def\qed{\rule{6pt}{6pt}}
\def\r1rn{{R_1,\dots,R_n}}
\def\r1rN{{R_1,\dots,R_N}}

\def\ri{\rightarrow}
\def\rr1n{{r_1,\dots,r_n}}

\def\s1sn{{S_1,\dots,S_n}}
\def\sbq{\subseteq}

\def\sqr{\square}

\def\t{\times}

\def\ti{\times}

\def\tx{\text}

\def\we{\wedge}

\def\x1xn{{x_1,\dots,x_n}}
\def\xx1n{{x,x_1,\dots,x_n}}

\def\l1{\L}

\def\Ri0{{R_{i_0}}}

\def\Rj0{{R_{j_0}}}

\def\vapi0{{\varphi_{i_0}}}

\def\vapj0{{\varphi_{j_0}}}

\def\lrarr{{\leftrightarrow}}

\def\r1rN{{R_1,\dots,R_N}}

\newlength{\JaneDepth}
\newlength{\JaneWidth}

\newcommand{\jane}[1]{%
  \settodepth{\JaneDepth}{\ensuremath{#1}}%
  \addtolength{\JaneDepth}{.4ex}%
  \settowidth{\JaneWidth}{\ensuremath{#1}}%
  \addtolength{\JaneWidth}{.75ex}%
  \ensuremath{%
    \mathop{%
      \makebox[0pt][l]{%
      \rule[-\the\JaneDepth]{.4pt}{.5ex}%
      \rule[-\the\JaneDepth]{\the\JaneWidth}{.4pt}%
      \rule[-\the\JaneDepth]{.4pt}{.5ex}}%
      \makebox[\JaneWidth][l]{%
        \hspace{.5ex}\ensuremath{#1}}%
      }%
  }%
}

\def\lastname{Shehtman, Shkatov}

\begin{document}
{
\binoppenalty = 10000
\relpenalty   = 10000

\renewcommand{\leftheadline}{\ifnum\value{page}=\hypergetpageref{FirstPage}\firstheadline
      \else\eightrm \thepage\quad\hfill\ifshttl\shtitle\else{Some prospects for semiproducts and products of modal logics}\fi\hfill\fi}

\begin{frontmatter}

  \title{Some prospects for semiproducts and products of modal logics}

  \author{Valentin Shehtman} \address{Institute for Information
    Transmission Problems, Russian Academy of Sciences \\ National
    Research University Higher School of Economics \\
    Moscow State University, Moscow, Russia}

  \author{Dmitry Shkatov}

  \address{School of Computer Science and Applied Mathematics, \\
    University of the Witwatersrand, Johannesburg, South Africa}

  \begin{abstract}
    We consider products and semiproducts of propositional modal
    logics $\bLa$ with $\bS5$ and present new examples of product and
    semiproduct logics axiomatized in the `minimal' way and enjoying
    the product (or semiproduct) FMP. An essential part of the proof
    is local tabularity of these (semi)products for $\bLa$ of finite
    depth; it is obtained by using bisimulation games. These results
    readily imply decidability for 1-variable fragments of predicate
    modal logics $\BQ\bLa$ and $\BQ\bLa$+Barcan formula.  We also
    present new counterexamples, i.e. (semi)products not axiomatizable
    in the simplest way.
  \end{abstract}

  \begin{keyword}
    modal logic, 1-variable fragment, product of modal logics,
    bisimulation game, finite model property
  \end{keyword}
\end{frontmatter}

\section{Introduction}
Semiproducts and products are special types of combined modal
logics. Their systematic investigation began in the 1990s, notably due
to connections with other areas of logic, both pure and applied,
cf. \cite{GKWZ}.  Nowadays the field has become even more interesting
and intriguing; for an overview of some developments cf. \cite{K08}.
In this note we are especially interested in (semi)products with
$\bS5$, due to their interpretation in modal
predicate logic translating the $\bS5$-necessity into the universal
quantifier.

One of the starting points in the study of products was the
``product-matching'' theorem (\cite{GKWZ}, Theorem 5.9) --- the
product of two Kripke complete Horn axiomatizable logics is
axiomatized in the minimal way. A similar result for semiproducts
(``semiproduct-matching'') is known for particular cases only (ibid.,
Theorem 9.10).  Here we present some new positive examples --- Horn
axiomatizable logics that are semiproduct-matching with $\bS5$ and
have the product finite model property (FMP). This implies
decidability and the FMP for corresponding 1-variable modal predicate
logics.

We also present new counterexamples --- two infinite families of logics not semiproduct-matching with $\bS5$. In particular, we show that
Horn axiomatizable complete logics may not be semiproduct-matching.

\section{Preliminaries}
We consider normal monomodal predicate logics, as defined in
\cite{GSS}, in a signature with predicate letters only.  A logic is a
set of formulas containing standard first-order axioms and the axiom
of $\BK$ and closed under standard rules (including predicate
substitution).  The minimal predicate extension of a propositional
monomodal logic $\bLa$ is denoted by $\BQ\bLa$; $\BQ\bLa\BC$ denotes
$\BQ\bLa+\forall x\, \sqr\, P(x) \ri \sqr\, \forall x\, P(x)\mbox{
  (the Barcan axiom)}.$

Formulas constructed from a single variable $x$ and monadic predicate
letters are called {\em 1-variable}. Formulas in which every
subformula of the form $\sqr B$ contains at most one parameter are
called {\em monodic} \cite{GKWZ}.

\begin{lemma}\label{L11}
  Every monadic monodic formula with at most one parameter is
  equivalent to a 1-variable formula in $\BQ\BK$.
\end{lemma}

In turn, every monomodal
 1-variable formula $A$ translates into a bimodal
propositional formula $A_*$ with modalities $\sqr$ and $\blacksquare$,
if every atom $P_i(x)$ is replaced with a proposition letter $p_i$ and
every quantifier $\fo x$ with $\blacksquare$.  The {\em 1-variable
  fragment} of a predicate logic $L$ is the set
$$L\!-\!1:=\{A_*\mid A\in L,~ A\mbox{ is 1-variable}\}.$$

For a modal predicate logic $L$, we have the
following:

\begin{lemma}\label{L12}
  $L\!-\!1$ is a bimodal
  propositional logic containing $\BK\sem$.
\end{lemma}

\begin{definition}
  The {\em product} of frames $F_1=(U_1,R_1),~F_2=(U_2,R_2)$ is\\
  $F_1\ti F_2:=(U_1\ti U_2,R_h,R_v)$, where
$$R_h(u,v)=R_1(u)\ti\{v\},~R_v(u,v)=\{u\}\ti R_2(v).$$
A {\em semiproduct} of $F_1$ and $F_2$ is
a subframe $(F_1\ti F_2)|W$ where $R_h(W)\sbq W$.
\end{definition}

Consider a monomodal propositional logic $\bLa$ (in the language with $\sqr$)
and $\bS5$ (in the language with $\bsq$). Put 
$$\bLa\,_{\_\!\_}\!|\,\bS5:=\bLa*\bS5+\sqr\bsq p\ri\bsq\sqr p, \quad
[\bLa,\bS5]:=\bLa\,_{\_\!\_}\!|\,\bS5+\bsq \sqr p\ri\sqr\bsq p,
$$
where $*$ denotes fusion. 

\begin{definition}
  The {\em product} $\bLa\ti \bS5$ is the logic of the class of all
  products of $\bLa$-frames with $\bS5$-frames. Similarly, the {\em
    semiproduct} $\bLa\rightthreetimes \bS5$ is the logic of the class
  of all semiproducts of such frames.
\end{definition}
In both cases, instead of arbitary $\bS5$-frames one can use single
clusters.

\begin{definition}
  The {\em Kripke-completion} $\overline{L}$ of a modal predicate
  logic $L$ is the logic of the class of all predicate Kripke frames
  validating $L$.
\end{definition}

\begin{lemma}
\begin{enumerate}
\item
$\bLa\sem\sbq \BQ\bLa\!-\!1\sbq\overline{\BQ\bLa}\!-\!1=\bLa\rightthreetimes \bS5$.
\item
$[\bLa,\bS5]\sbq \BQ\bLa\BC\!-\!1\sbq\overline{\BQ\bLa\bC}\!-\!1=\bLa\ti\bS5$.
\end{enumerate} 
\end{lemma}

\begin{definition}
  $\bLa$ and $\bS5$ are called {\em semiproduct-matching} if
  $\bLa\sem=\bLa\rtt\bS5$ and {\em product-matching} if
  $[\bLa,\bS5]=\bLa\ti\bS5$.

  $\bLa$ is called {\em quantifier-friendly}, if
  $\BQ\bLa\!-\!1=\bLa\sem$, and {\em Barcan-friendly}, if
  $\BQ\bLa\BC\!-\!1=\bLa\ti\bS5$.
\end{definition}

So $\bLa$ is quantifier-friendly (respectively, Barcan-friendly)
whenever $\bLa$ and $\bS5$ are semiproduct-matching (respectively, product-matching).

\begin{theorem}\label{T28} (cf. \cite{GKWZ}, Theorem 5.9).  If $\bLa$ is Kripke
  complete and Horn axiomatizable, then $\bLa$ and $\bS5$ are
  product-matching.
\end{theorem}

For semiproducts an analogue of this theorem does not hold (see below). Let us recall, in a
slightly more general form, a number of positive results presented in
\cite{GKWZ}, Theorem 9.10.\footnote{In \cite{GKWZ} semiproducts are called
  `expanding relativized products', $\bLa\sem$ is denoted by
  $[\bLa,\bS5]^{EX}$, $\bLa\rtt\bS5$ by $(\bLa\ti\bS5)^{EX}$.}

\begin{definition}
A {\em one-way PTC-logic} is a modal propositional logic axiomatized by formulas of the form $\sqr p\ri\sqr^n p$ and variable-free formulas. 
\end{definition}

\begin{theorem}\label{T6} 
$\bLa$ and $\bS5$ are semiproduct-matching for any one-way PTC-logic $\bLa$.
\end{theorem}

\section{Counterexamples}

\begin{theorem}(cf. \cite{P45SS2})
Let
$$\sqr\BT:=\BK+\sqr(\sqr p\ri p), \quad
\bSL4:=\bbK4+\Di p\,\lrarr\,\sqr p.$$
If $\sqr\BT\sbq\bLa\sbq\bSL4$, then $\bLa$ and $\bS5$ are not
semiproduct-matching.
\end{theorem}

For the proof note that
$\sqr\bsq(\sqr p\ri p)\in(\bLa\rtt \bS5)-(\bLa\sem)$.

\smallskip
Hence we obtain counterexamples to an analogue of Theorem \ref{T28}: Horn axiomatizable logics $\sqr\BT$, ${\bf K5}$, ${\bf K45}$ are not  semiproduct-matching with $\bS5$.

Nevertheless, we have

\begin{remark}
  (cf. \cite{P48abs4}) {\em Every complete Horn axiomatizable logic is quantifier-friendly.}
\end{remark}

\begin{theorem}\label{Tneg}
  If $\BK+Alt_n\sbq\bLa\sbq\BK+Alt_n+\sqr^m\bot$ for $n\geq 3$,
  $m\geq 2$, then
  $\bLa$ and $\bS5$ are neither product- nor semiproduct-matching.
\end{theorem}

\begin{proof} (Sketch.) Take the product $F_1\ti F_2$, where $F_1$ is
  the irreflexive tree with the root $0$ and the leaves $1,\ldots,n$
  and $F_2$ is the two-element cluster $\{1,2\}$; replace $R_v$ by the
  least equivalence relation $S_2$ such $(x,y)S_2(x',y')$ for $x=x'=0$
  or $x=x'>3$,
  $(1,1)S_2(2,2), ~(1,1)S_2(3,2),~(1,2)S_2(2,1), ~(1,2)S_2(3,1)$. The
  resulting frame $G_n$ is not a p-morphic image of a semiproduct of a
  $(\BK+Alt_n)$-frame and a cluster while
  $G_n\mo [\BK+Alt_n+\sqr^2\bot,\bS5]$. Therefore its Fine-Jankov
  formula belongs to $(\bLa\rtt \bS5)-[\bLa,\bS5]$.
\end{proof}

A standard canonical model argument proves Kripke-completeness
of all the logics $\BQ\bLa$ for $\bLa=\BK+Alt_n,~\BK+Alt_n+\sqr^m\bot$. So we obtain
\begin{corollary}
The logics $\BK+Alt_n$, $\BK+Alt_n+\sqr^m\bot$ are not quantifier-friendly for $n\geq 3$, $m\geq2$.
\end{corollary}



\section{Local tabularity}
Recall that a propositional logic $L$ is {\em locally tabular}, if for
any finite $k$ there exist finitely many $L$-non-equivalent formulas in
$k$ proposition letters.
 
It is well known that every extension of a locally tabular modal logic
in the same language is locally tabular; every locally tabular logic
has the FMP.

\begin{theorem}\label{T81}
Every logic $(\BK+\sqr^n\bot)\sem$ is locally tabular.
\end{theorem}

This theorem is proved by using bisimulation games; the corresponding technique is described in \cite{p42}.



A monomodal logic $\bLa$ is {\em of finite depth} if $\sqr^n\bot\in\bLa$ for some $n$.

\begin{corollary}
If $\bLa$ is of finite depth, then the logics 
$\bLa\rtt\bS5, ~\bLa\sem$ have the FMP; so their finite axiomatizability implies decidability.
\end{corollary}

In particular, $\bLa\rtt\bS5$ ($\bLa\ti\bS5$) is decidable, provided
$\bLa$, $\bS5$ are semiproduct- (product-) matching and $\bLa$ is of finite depth.

\section{More examples of semiproduct-matching}
In contrast with Theorem \ref{Tneg}, we can identify some other
logics that are semiproduct-matching with $\bS5$.

\begin{lemma}
Consider the axiom $Ath:=\Di\Di p\ri \sqr\Di p$. $Ath$-frames are defined by the following first-order condition:
$$\fo x,y,z,u\, (xRy \we xRz \we yRu \ri zRu).$$
\end{lemma}
We call these frames {\em thick}.

\begin{proposition}\label{SPM}
The logics $\BK+Ath$, $\BK+Ath+\sqr^n\bot$ for $n\geq 1$ are semiproduct-matching with $\bS5$.
\end{proposition}

\begin{proof} (Sketch.)  Every countable rooted $\BK\sem$-frame $H$ is
  a p-morphic image of a semiproduct $G$ of a tree $F$ and a cluster
  $C$; the proof is similar to the one for products,
  cf. \cite{GKWZ}. Since $Ath$ is a Horn formula, we can take the
  corresponding Horn closure $G^+$; then $G^+$ is a semiproduct of
  $F^+$ and $C$. If $H\mo Ath$, we obtain a p-morphism from $G^+$ onto
  $H$. So every formula refutable on $H$ is not in
  $(\BK+Ath)\rtt\bS5$.

Adding variable-free axioms   $\sqr^n\bot$ does not affect this argument.
\end{proof}

\section{Product and semiproduct FMP}
In many cases (semi)products enjoy the (semi)product FMP.  In
particular, if $L_1$ is tabular and $L_2$ has the FMP, then
$L_1\ti L_2$ has the product FMP \cite[Cor. 5.9]{p37}. 
Probably, this may not be true, if $L_1$ is only locally tabular.  Examples of
semiproduct FMP can be found in \cite{NPR}, but they do not cover our
next result:

\begin{theorem}\label{Tmain} 
  For $\bLa= \BK+Ath$ and $\bLa=~\BK+\sqr^n\bot+Ath$, the
  (semi)product of $\bLa$ with $\bS5$ has the (semi)product FMP.
\end{theorem}

\begin{corollary}
  For logics $\bLa$ from Theorem \ref{Tmain} $\BQ\bLa-1$ has the FMP,
  i.e., is complete w.r.t. finite Kripke frames with finite domains.
\end{corollary}


Let us give some comments about the proof of Theorem \ref{Tmain} for
the case of semiproducts. Note that
$(\BK+Ath)\rtt\bS5=\bigcap_n((\BK+\sqr^n\bot+Ath)\rtt\bS5)$, so it
suffices to consider only $L=\bLa\rtt\bS5$ for
$\bLa=\BK+\sqr^n\bot+Ath$ and show that every finite rooted $L$-frame
$F=(W,R_1,R_2)$ is a p-morphic image of a finite semiproduct of a
$\bLa$-frame with a cluster.  A {\em row} in $F$ is a connected
component in $(W,R_1)$; a {\em column} is an equivalence class under
$R_2$; a {\em block} is a non-empty intersection of a row and a
column.  $F$ is {\em straight} if all its blocks are singletons. We
can show that $F$ is a p-morphic image of a straight rooted $L$-frame
isomorphic to a semiproduct of a $\bLa$-frame and a cluster.

\begin{remark}
  We hope our main results can be transferred to extensions of
  ${\bf GL}$. The logic $\bGL\sem$ is the well-known provability logic
  of Artemov--Japaridze, which is semiproduct-matching with $\bS5$. A
  transitive analogue of $Ath$ is R. Solovay's axiom
  $AS:=\sqr(\sqr p\ri\sqr q)\vee\sqr(\sqr q\ri p\we\sqr p)$.  We may
  conjecture that ${\bf SOL}:=\bGL+AS$ (Solovay's logic of
  ``provability w.r.t $\BZ\BF$'' cf.~\cite{Boolos}, ch. 13) is also
  semiproduct-matching with $\bS5$ and that ${\bf SOL}\rtt\bS5$ has
  the semiproduct FMP.
\end{remark}


\end{document}